\documentclass[onecolumn]{autart}

\usepackage{natbib}
\usepackage{amsmath,amsfonts,amssymb}
\usepackage{graphicx,subfigure}

\usepackage{float}
\usepackage{epsf}
\usepackage{fancyhdr}
\usepackage{subfigure}
\usepackage{latexsym}

\usepackage{tikz}  

\usepackage{color}
\definecolor{Royalblue}{cmyk}{1,0.30,0.2,0.2}

\definecolor{purple}{rgb}{0.62, 0.0, 0.77}
\newcommand{\luci}{\color{black}}

\newcommand{\nd}{\noindent}

\newcommand{\Tr} {\mbox{\rm tr}}

\newcommand{\bmat}{\left[ \begin{matrix}}
\newcommand{\emat}{\end{matrix} \right]}

\newcommand{\beq}{\begin{equation}}
\newcommand{\eeq}{\end{equation}}
\newcommand{\tp}{^{\top}}
\newcommand{\bea}{\begin{eqnarray}}
\newcommand{\eea}{\end{eqnarray}}

\newcommand\numberthis{\addtocounter{equation}{1}\tag{\theequation}}

\newcommand{\E}{{\mathbb E}\,}

\newcommand{\Rbb}{\mathbb R}
\newcommand{\Cbb}{\mathbb C}

\newcommand{\Zbb}{\mathbb Z}
\newcommand{\Nbb}{\mathbb N}
\newcommand{\Tbb}{\mathbb T}

\newcommand{\yb}{\mathbf  y}
\newcommand{\zb}{\mathbf  z}

\newcommand{\eb}{\mathbf  e}

\newcommand{\ub}{\mathbf  u}

\newcommand{\tb}{\mathbf  t}
\newcommand{\ssb}{\mathbf  s}
\newcommand{\kb}{\mathbf  k}

\newcommand{\Mb}{\mathbf M}
\newcommand{\Nb}{\mathbf N}

\newcommand{\Rb}{\mathbf R}
\newcommand{\Sb}{\mathbf S}

\newcommand{\Wb}{\mathbf W}

\def\sigmab{\boldsymbol{\sigma}}
\def\omegab{\boldsymbol{\omega}}
\def\alphab{\boldsymbol{\alpha}}
\def\Phib{\boldsymbol{\Phi}}
\def\thetab{\boldsymbol{\theta}}
\def\Deltab{\boldsymbol{\Delta}}
\def\zerob{\boldsymbol{0}}

\newtheorem{theorem}{Theorem}

\newtheorem{remark}{Remark}

\begin{document}

\begin{frontmatter}

\title{Mean-square consistency of the $f$-truncated $\text{M}^2$-periodogram }

\author[Unipd]{Lucia Falconi}\ead{lucia.falconi@phd.unipd.it},
\author[Unipd]{Augusto Ferrante}\ead{augusto@dei.unipd.it},
\author[Unipd]{Mattia Zorzi}\ead{zorzimat@dei.unipd.it}
     
\address[Unipd]{Department of Information Engineering, University of Padova, Via Gradenigo 6/B, 35131 Padova, Italy} 
      
\begin{keyword}  
Consistency, Periodogram, Spectral estimation                                              
\end{keyword}

\begin{abstract}                          
	The paper deals with the problem of estimating the M$^2$ (\textit{i.e.} multivariate and multidimensional) spectral density function of a stationary random process or random field. 
	We propose the $f$-truncated periodogram, \textit{i.e.} a truncated periodogram where the truncation point is a suitable function $f$ of the sample size.
	We discuss the asymptotic consistency of the estimator and we provide three concrete problems that can be solved using the proposed approach. Simulation results show the effectiveness of the procedure. 
\end{abstract}

\end{frontmatter}

\section{Introduction} 

The problem of estimating the spectral density of a stationary random process
from a finite length record of observed data is of paramount importance in control, signal processing and time-series analysis. 
Just to highlight a few of the countless important applications where spectral estimation plays a key role,
we mention the problem of estimating the impulse response of a SISO LTI system through ETFE, \cite{Ljung_SI,SODERSTROM_STOICA_1988}, the reconstruction of the topology of dynamic network, \cite{ID_DAHLHAUS,Avventi_ARMA,alpago2021scalable,zorzi2020autoregressive} and targets detection
using radar signals, \cite{Engels2017,Engels2014}.
Spectral estimation has been extensively studied and many different approaches, both parametric and non parametric, have been proposed;
we refer the reader to  \cite{stoica2005spectral}  for an overview of the literature and a rich list of references.

The most basic and widely used estimation method is the periodogram that can be regarded as the classical mean for spectral estimation.  Even more complex and refined procedures such as the {\em THREE} and {\em THREE-like} methods start from  a coarse estimate often given by the periodogram, see \cite{byrnes2000new,FERRANTE_TIME_AND_SPECTRAL_2012,Hellinger_Ferrante_Pavon,BETA,georgiou2006,zhu_SIAM} and the references therein.
The statistical properties of the periodogram have been extensively studied over the years. It is well known that the periodogram provides an asymptotically unbiased estimate of the underlying spectrum, but it has some problems: indeed, the periodogram is an inconsistent spectral estimator\footnote{Recall that an estimator is said to be  consistent  if the mean square error of the estimator tends to zero as the sample size tends to infinity.} so that, contrarily to what one would expect, the quality of the periodogram estimate does not get better as the number of samples increases.  
Furthermore, the periodogram values at adjoining frequencies are asymptotically  uncorrelated, so that the estimate exhibits an erratic behavior.
Finally, a crucial issue is enforcing positivity of the estimated spectral density which can be achieved by the so-called biased version of the periodogram. The latter, however, forces the rank of the estimated  spectral density to be equal to one which is a serious problem in many situations where multivariate processes are considered both in the unidimensional and in the multidimensional case.

A traditional way of dealing with these problems is to smooth the periodogram across frequencies.
The first reference to the idea of smoothing the periodogram to obtain a better spectral estimate appears to be \cite{daniell1946discussion}
that proposed to reduce the large variance of the basic periodogram estimator by averaging the periodogram over small intervals centered on the current frequency.
Since then, several methods to smooth the periodogram have been considered \cite{Bartlett1948}, \cite{Bartlett1955}, \cite{Blackman1958}. 

A viable and simple strategy to smooth the periodogram consists in truncating the sum in its definition formula,
\emph{i.e.} in windowing the periodogram  with a rectangular weight function.
The effect of the truncation is to reduce the variance of the estimator at the price of a higher bias. Hence,
a proper choice of the truncation point is crucial to correctly balance the tradeoff between bias and variance. This choice, however,  is far from being obvious.

In this paper, the interest is on the class of discrete-time, second-order, stationary, Gaussian, possibly multivariate and multidimensional random processes.  
For this class of models, we propose a simple procedure for spectral estimation with asymptotic performance guarantees. It consists of building a truncated periodogram and choosing the truncation point as a suitable function 
of the sample size. We prove that, under the only assumption that the spectral density of the underlying process has absolutely summable Fourier coefficients,   our truncation method provides an estimator which is mean-square consistent both when 
applied with the  unbiased and biased sample covariances.
In the interest of notation simplicity, the proofs are given for real-valued signals; the generalization to the complex case is straightforward.

Results in the same vein have been discussed in the past literature, mostly dealing with scalar and unidimensional random processes (see for example \cite{grenander1953, Lomnicki1957, Parzen1957}).
In particular, we refer the reader to the classical textbook \cite{priestley1982spectral} for a comprehensive review of the results.
Priestley considers a general class of windowed periodogram and analyzes its asymptotic statistical properties by working with a representation of the estimator as a weighted integral of the full periodogram.  Consistency is guaranteed provided that the impulse response of the filter generating the process in question decays to zero strictly faster than $ \frac{1}{t^c}$ with $c>2$.  
The result is then extended to multidimensional, but scalar, case. \\
In our paper, focusing the attention to rectangular window functions, we provide a unified, direct and conceptually simple and self-contained proof of the consistency of the $f$-truncated periodogram that holds for multidimensional and multivariate random processes
under the weaker assumption that the filter generating the process is only BIBO stable.
We also show that, at the price of considering consistence in the $L_2$ sense
(instead of pointwise), our result continues to hold under the even weaker assumption
that the impulse response of the filter generating the process has finite energy.
Our approach is different from \cite{priestley1982spectral} as we work directly with the expression (19) of the spectral estimator and  we then analyze its consistency by studying the sampling properties of the {\luci sample} autocovariances.

The paper is organized as follows. Separate discussions are given for unidimensional random processes (Section \ref{sec:1d}) and multidimensional random fields (Section \ref{sec:multidim}). Although the methodology is the same, the unidimensional case is analyzed to better understand the paradigm as well as to highlight its importance in practice through some examples: in {\luci Subsection } \ref{sec:ETFE} we consider the impulse response estimation problem for a SISO system; in {\luci Subsection }  \ref{sec:NET} we consider the problem of learning a dynamic Gaussian graphical model. We also provide an example for the multidimensional case, that is the target parameter estimation problem in automotive radars, see {\luci Subsection } \ref{sec:RAD}. 
Finally, Section \ref{sec:conc} concludes the paper.

\vspace{\baselineskip}

\nd{\em Notations:} 
In the following, $\E$ denotes the mathematical expectation, $\Zbb$ the set of integers, $\Rbb$ the real line and $\Cbb$ the complex plane.
Given a matrix $\Mb$, $\Mb\tp$ is its transpose and $M_{ij} $ is the element of $M$ in the $i$-th row and $j$-th column.  
The notation $(\cdot)^*$ means the complex conjugate transpose. {\luci If $\Mb$ is a square matrix, $\Tr(\Mb)$ denotes its trace.}
The symbol $\Vert \cdot \Vert$ may denote both the Frobenius norm of a matrix or the  $L_1$-norm {\luci of} a function depending on the context.

\section{Unidimensional case} \label{sec:1d}
Suppose that we have the $m$-valued, zero-mean, stationary random process
\beq \label{eq:model_1d} 
\yb (t) = \sum_{\sigma = 0}^{\infty} \Mb (\sigma) \eb (t - \sigma), \qquad t \in  \Zbb,
\eeq 
where $\eb = \{ \eb(t), \;  t\in \Zbb \}$ is normalized white Gaussian noise of dimension $p.$ The impulse response $ \Mb (\cdot) : \Zbb \to \Rbb^{m \times p} $ is such that $ \sum_{\sigma = -\infty}^{\infty} \Vert \Mb (\sigma ) \Vert = \sum_{\sigma = 0}^{\infty} \Vert \Mb (\sigma ) \Vert  < \infty, $
meaning that the system is causal and BIBO stable, which is the case in most practical applications. \\
\begin{remark}
	In \eqref{eq:model_1d} we have assumed that the system  generating the process $\yb (t)$ is causal, i.e. $\Mb (\sigma) =0$ for all $\sigma<0$, because this is the most common situation in the unidimensional case, where $t$ typically represents the time variable. However, this restriction is not necessary: in fact the asymptotic consistency of the proposed spectral estimator is proved in Section \ref{sec:multidim} for possibly non-causal systems. In particular, our results still hold when the system 
	is not causal and the sum in \eqref{eq:model_1d} ranges from $-\infty$ to $+\infty$.
\end{remark} 
It is well-known that the autocovariance sequence 
$ \Rb_k := \E \yb (t + k) \yb (t)\tp $
depends only on the lag $k$ and it enjoys the property $\Rb_k = \Rb_{-k}\tp. $ 
Define the spectrum of the process as the Fourier transform of the covariance function: 
\beq \label{eq:spectrum_1d}
\Phib (e^{i \theta}) : =  \sum_{k= -\infty}^{k=\infty} \Rb_k
e^{-i  \theta k  } , \quad \theta \in [0, 2\pi).
\eeq
Then, a simple estimator of the spectrum is:
\beq \label{eq:periodogram_1d}
\hat \Phib (e^{i \theta}) : =  \sum_{k=-n}^{k=n} \hat \Rb_k e^{-i \theta k}
\eeq
where $\hat \Rb_k$ is an estimate of the covariance lag $\Rb_k$ obtained from the available
sample  $ \{  \yb (t), t = 1 , ... , N \}. $ 
There are two standard ways to obtain the sample covariances required in \eqref{eq:periodogram_1d}: 
the unbiased estimate
\beq \label{eq:unbiasedACS}
\hat \Rb_k := \frac{1}{N-k} \sum_{t=1}^{N-k} \yb (t + k) \yb (t)\tp \quad k=0,\dots, n,
\eeq
and the biased sample covariance
\beq \label{eq:biasedACS}
\hat \Rb_k := \frac{1}{N} \sum_{t=1}^{N-k} \yb (t + k) \yb (t)\tp \quad k=0,\dots, n.
\eeq

Next, we address the fundamental problem of selecting the parameter $n$ in \eqref{eq:periodogram_1d}. 
A typical choice is $n=N-1$, in which case the estimator \eqref{eq:periodogram_1d} is called the periodogram. 
The main problem with the periodogram lies in its large variations about the true spectrum, even for very large data samples. 
This effect can be reduced by truncating the periodogram, \emph{i.e.} by choosing $n < N-1.$ 
We may expect that the smaller the $n$, the larger the reduction in variance and the lower the resolution. Hence, the choice of $n$ should be based on a trade-off between spectral resolution and statistical variance. Here, we propose to select the length $n$ of the truncated periodogram \eqref{eq:periodogram_1d} as a function $f(N)$ of the numerosity $N$ of the data sample. Clearly, the performance  of the obtained estimator is determined by the chosen function $f;$ to stress this dependence the estimator will hereafter be referred to as $f$-truncated periodogram.
If 
\beq \label{eq:n}
n :=   f(N),
\eeq
where $f$ is a function taking values in $\Nbb$ such that 
\begin{align}
	\lim_{N \to \infty}  f(N) &= \infty,	\label{eq:limit_1} \\
	\lim_{N \to \infty}  \frac{f(N)^2}{N} &= 0	\label{eq:limit},
\end{align} 
then the following theorem can be stated:
\begin{theorem} \label{th:cons_1d}
	Given any stochastic  process   $\yb$ of the form \eqref{eq:model_1d},
	the $f$-truncated periodogram \eqref{eq:periodogram_1d} with $n$ defined by \eqref{eq:n}-\eqref{eq:limit_1}-\eqref{eq:limit} and $\hat \Rb_k$ estimated through  \eqref{eq:unbiasedACS} or \eqref{eq:biasedACS} is a uniformly mean-square consistent estimator of the spectral density, that is
	$$ 
	\lim_{N \to \infty} \; \E  \Vert \Deltab (e^{i \theta}) \Vert^2  = 0  \quad \text{uniformly over } [0,2\pi),
	$$
	where $ \Deltab(e^{i \theta}) := \Phib(e^{i \theta}) - \hat \Phib (e^{i \theta}).  $ 
\end{theorem}
The proof is deferred to Section \ref{sec:multidim}, where the result is stated in the more general setting of a multidimensional and multivariate random field (see Theorem  \ref{th:cons_unbias} for the case in which unbiased estimates of the covariance sequence are considered and Theorem \ref{th:cons_bias} for the biased estimates case).

\begin{remark}
When studying the linear relation between two jointly stationary signals $\yb = \{ \yb(t), \;  t\in \Zbb \}$ and $\ub = \{ \ub(t), \;  t\in \Zbb \}$, it is sometimes useful to estimate the cross spectrum $$ \Phib_{\yb \ub}(e^{i \theta}) := \sum_{k=-\infty}^{\infty} \Rb_{\yb \ub,k} e^{-i\theta k} $$
where $ \Rb_{\yb \ub,k} = \E \yb(t+k)\ub(t)\tp.$ \\
To this end, define $ \zb (t) := [ \yb(t)\tp \; \; \ub(t)\tp]\tp$ and let $\hat\Phib_\zb $ denote the estimator of the spectral density $\Phib_\zb.$
If we partition $\hat\Phib_\zb $ as 
\beq \label{eq:joint_spectrum} 
\hat\Phib_\zb = \begin{bmatrix} 
	\hat\Phib_{\yb} & \hat\Phib_{\yb \ub}   \\
	\hat\Phib^*_{\yb \ub} & \hat\Phib_{\ub}   \end{bmatrix}, \eeq
we immediately see that an estimate of the cross-spectrum $\Phib_{\yb \ub}$ may be obtained from the corresponding block in $\hat\Phib_\zb.$ Moreover, it is clear that the mean square consistency of $\hat \Phib_\zb$ implies the  mean-square consistency of the cross-spectrum estimate obtained from \eqref{eq:joint_spectrum}.
\end{remark}

In the following, we give examples of two problems that can be solved with the proposed approach. 


\subsection{Impulse response estimation} \label{sec:ETFE}
The proposed $f$-truncated periodogram can be used for the reconstruction of a SISO system with impulse response  $ g(t), \; t\in \Zbb,$ with a non-parametric frequency-domain technique. 
The identification problem consists of estimating the unknown function $g$ starting from a known input  $u$  and $N$ noisy measurements
\beq \label{eq:etfe_model} 
y (t) = \sum_{\sigma = 0}^{\infty} g(\sigma) u(t - \sigma) + v(t) \qquad t=1,... ,N,
\eeq
where $v = \{ v(t), \; t \in \Zbb \}$ is white noise.

We proceed by first estimating the transfer function $G(e^{i\theta})$ of the system. The latter task is traditionally performed by means of the Empirical Transfer Function Estimate (ETFE), see \cite[chap. 6.4]{Ljung_SI}:
\beq \label{eq:Getfe}
\hat G_{ETFE} (e^{i\theta}) = \frac{Y_N(e^{i\theta}) }{U_N(e^{i\theta}) },
\eeq
where $Y_N $ and $U_N $ are the Fourier transforms of the input and output sequences, respectively. However, in most practical situations the ETFE estimate is not accurate even for large data samples, since its variance does not decrease with $N$ and the estimates at different frequencies are asymptotically uncorrelated. In the extreme case, it may happen that $U_N(e^{i\theta})$ is zero for certain values of $\theta$ making the {\luci ratio} in (\ref{eq:Getfe}) not well defined in that frequencies. Notice that these properties are closely related to those of periodogram estimates of spectra.
As an alternative to overcome these issues, one can consider the smoothed version of the ETFE: 
\beq \label{eq:etfe_G1}
\hat G_1 (e^{i\theta})  = \frac{ \hat \Phi_{{yu}_1}(e^{i\theta}) }{ \hat \Phi_{{u}_1}(e^{i\theta}) },
\eeq
where the spectral estimates $\hat \Phi_{u_1}$ and $\hat \Phi_{yu_1}$ are the $f$-truncated periodograms corresponding to $n=f_1(N)$ where $f_1$ satisfies the assumptions \eqref{eq:limit_1}-\eqref{eq:limit}. It is also worth noting that, under the assumption that the true spectrum of the input $u$ is coercive, the consistency property guarantees that the probability that denominator of (\ref{eq:etfe_G1}) is different from zero pointwise tends to 1 as  $N$ approaches infinity.\\
Then, given an estimate $\hat G (e^{i\theta})$ of the transfer function of the system, we easily take the representation of the signal back to the time domain through the inverse Discrete Time Fourier Transform.

As a benchmark we consider the reconstruction of the infinite-dimensional system with impulse response
$$ g(t) = \left(1 + 25 \left( \frac{t-20}{20}\right)^2 \right)^{-1} \quad  t = 0,1,2, \dots  $$
which is a translated and scaled version of the widely-known Runge function.
The standard deviation of the measurement noise $ v $ is $1\%$ of the maximum absolute value of the generated noiseless output samples. 
We assume that the system is initially at rest and we use the output of the low-pass filter $ H(z) = \frac{z + 0.35}{z-0.45} $ fed by a normalized white Gaussian noise as test input.
Given a data sample of numerosity $N$, we compute the transfer function estimate \eqref{eq:etfe_G1}
where  the covariance lags are estimated by \eqref{eq:biasedACS} and $ f_1(N) = \lfloor  \sqrt[3] N \rfloor.$
For the sake of comparison, we also consider the estimate 
$$\hat G_2 (e^{i\theta}) = \frac{\hat\Phi_{{yu}_2}(e^{i\theta})}{ \hat\Phi_{{u}_2}(e^{i\theta})},$$ where $\hat \Phi_{u_2} $ and $ \hat \Phi_{yu_2} $ are now the periodograms truncated to $ f_2(N) = \lfloor 0.01 N \rfloor$, and the raw ETFE estimate \eqref{eq:Getfe} as implemented  in the MATLAB System Identification Toolbox.
We hasten to remark that the function $f_2(N)$ does not satisfies the assumptions of Theorem \ref{th:cons_1d}. \\
For each value of the data size $N$, we perform a Monte Carlo simulation of $50$ trials. In each trial we randomly generate  the data sample $u$ and $y$.  We define the relative reconstruction error at the $j$-th MC run as
$$
err_j := \sqrt{\frac { \sum_{t=0}^{300} (g(t) - \hat g (t))^2 } {\sum_{t=0}^{300} g(t)^2  } },
$$
where $\hat g$ in one of the previous impulse response estimates, 
and the average identification error as
$$
\overline{err} : = \frac{\sum_{j=1}^{50} err_j}{50}.
$$
The results of the simulations are summarized in Figure \ref{fig:etfe_error}, \ref{fig:etfe_error_box} and \ref{fig:etfe_tf}:
Figure \ref{fig:etfe_error} plots the average error $\overline{err}$ of the three different estimators as a function of the data size $N$, 
Figure \ref{fig:etfe_error_box} the boxplots of the relative reconstruction errors for different values of $N$, 
and finally Figure  \ref{fig:etfe_tf} shows a typical realization of the absolute value of transfer function estimates $\hat G $, for $N = 2 \cdot 10^7$. 
The figures reveal that $\hat G_1 $ asymptotically outperforms the other two estimators. As a matter of fact, its relative reconstruction error approaches zeros as $N$ grows to infinity, in line with the fact that $\hat \Phi_{u_1} $ and $ \hat \Phi_{yu_1}$ are mean-square consistent spectral estimators. On the other hand, we have no theoretical guarantees on the asymptotic behavior of  $\hat \Phi_{u_2} $ and $ \hat \Phi_{yu_2};$ moreover, as already known, the raw ETFE generally provides a very poor estimate of the underlying system.

\begin{figure}
	\centering
	\includegraphics[width=\linewidth]{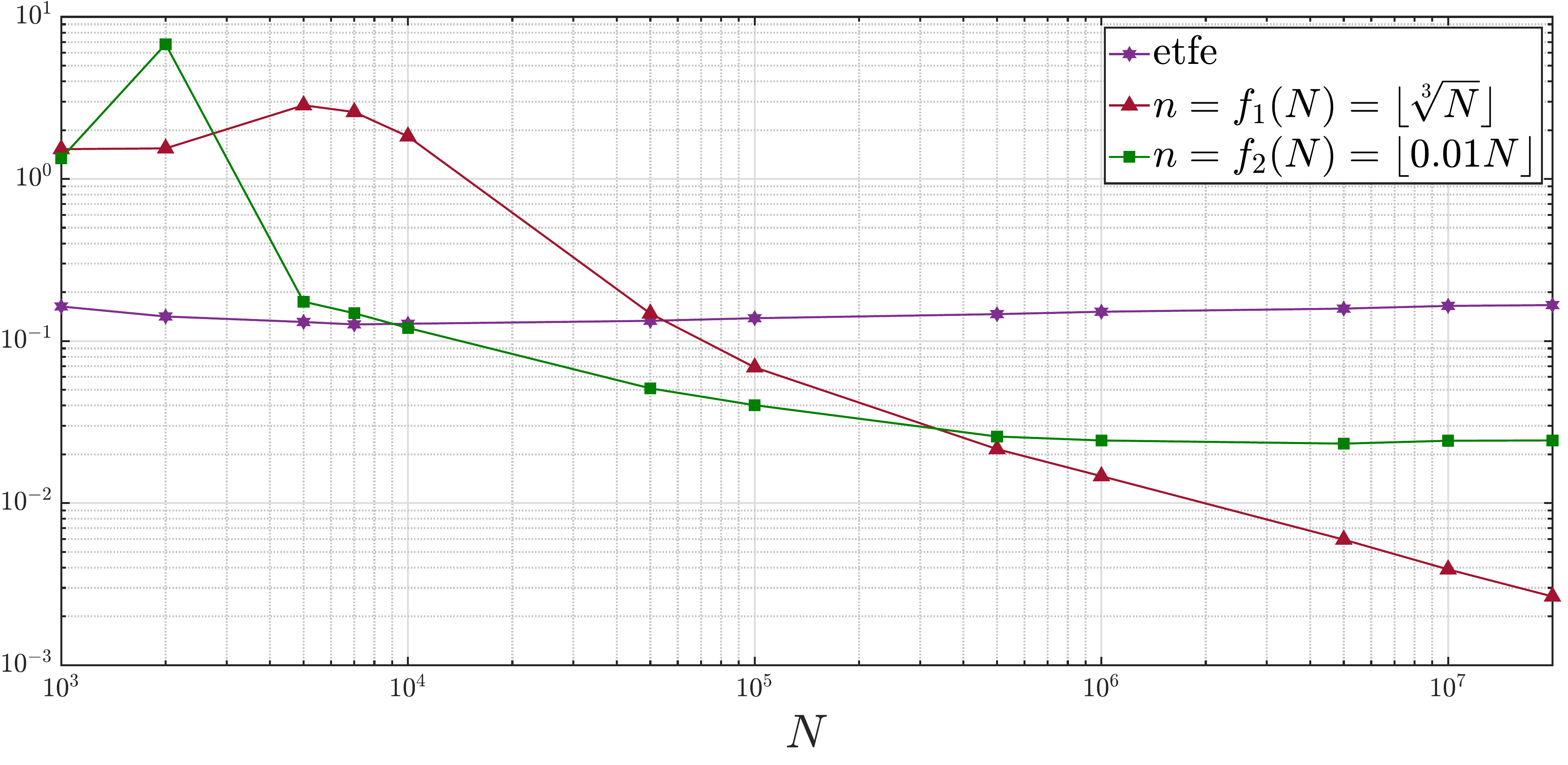}
	\caption{Impulse response estimation problem: comparison of the average identification error $\overline{err}$ of $\hat g_i (t)$, $i = 1,2,3$, for increasing values of $N$.  }
	\label{fig:etfe_error}
\end{figure}

\begin{figure}
	\centering
	\includegraphics[width=\linewidth]{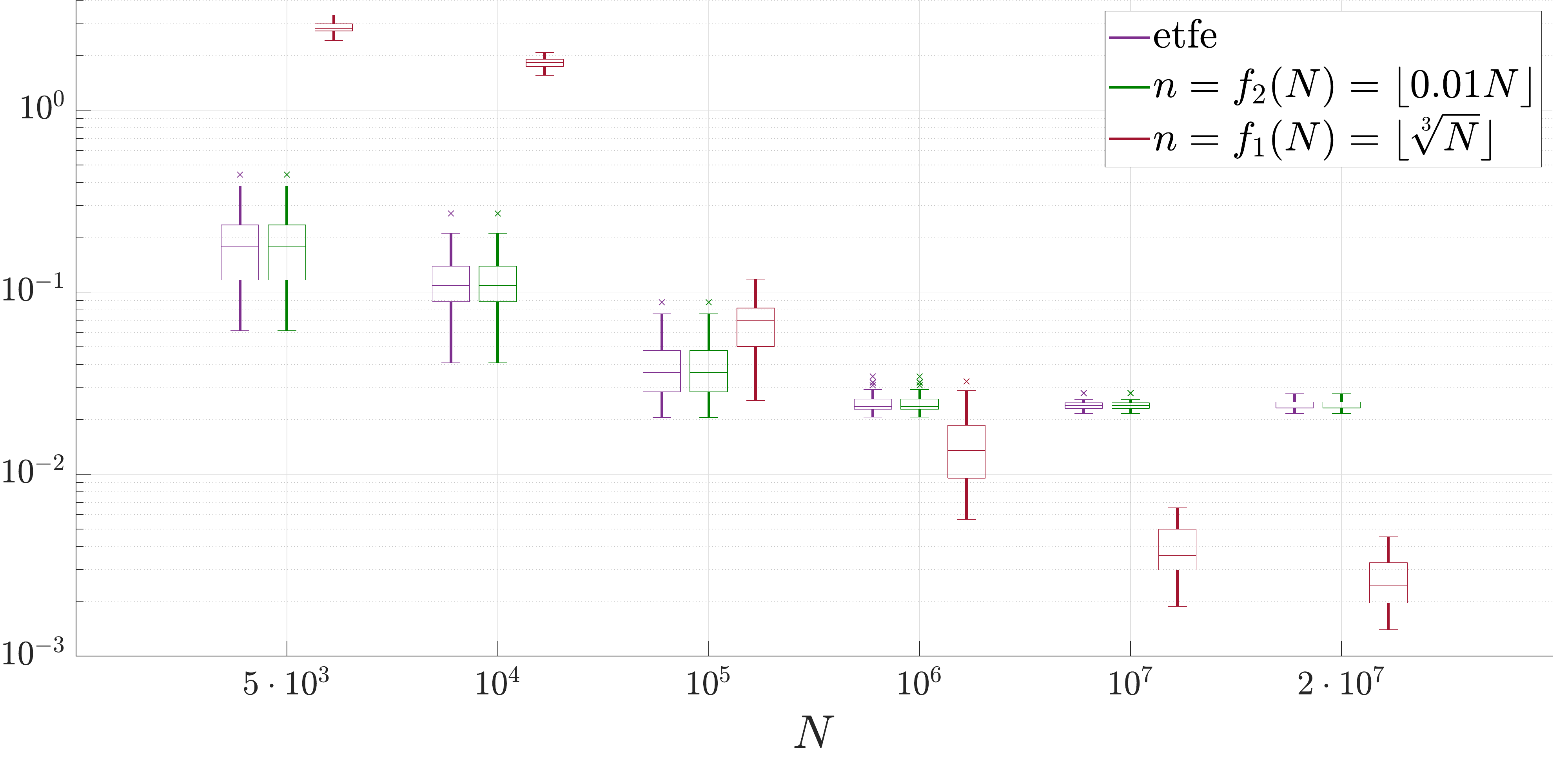}
	\caption{Impulse response estimation problem: comparison of the boxplots of the identification error  of $\hat g_i (t)$, $i = 1,2,3$, for the 50 MonteCarlo simulations for increasing values of $N$.  }
	\label{fig:etfe_error_box}
\end{figure} 

\begin{figure}
	\centering
	\includegraphics[width=\linewidth]{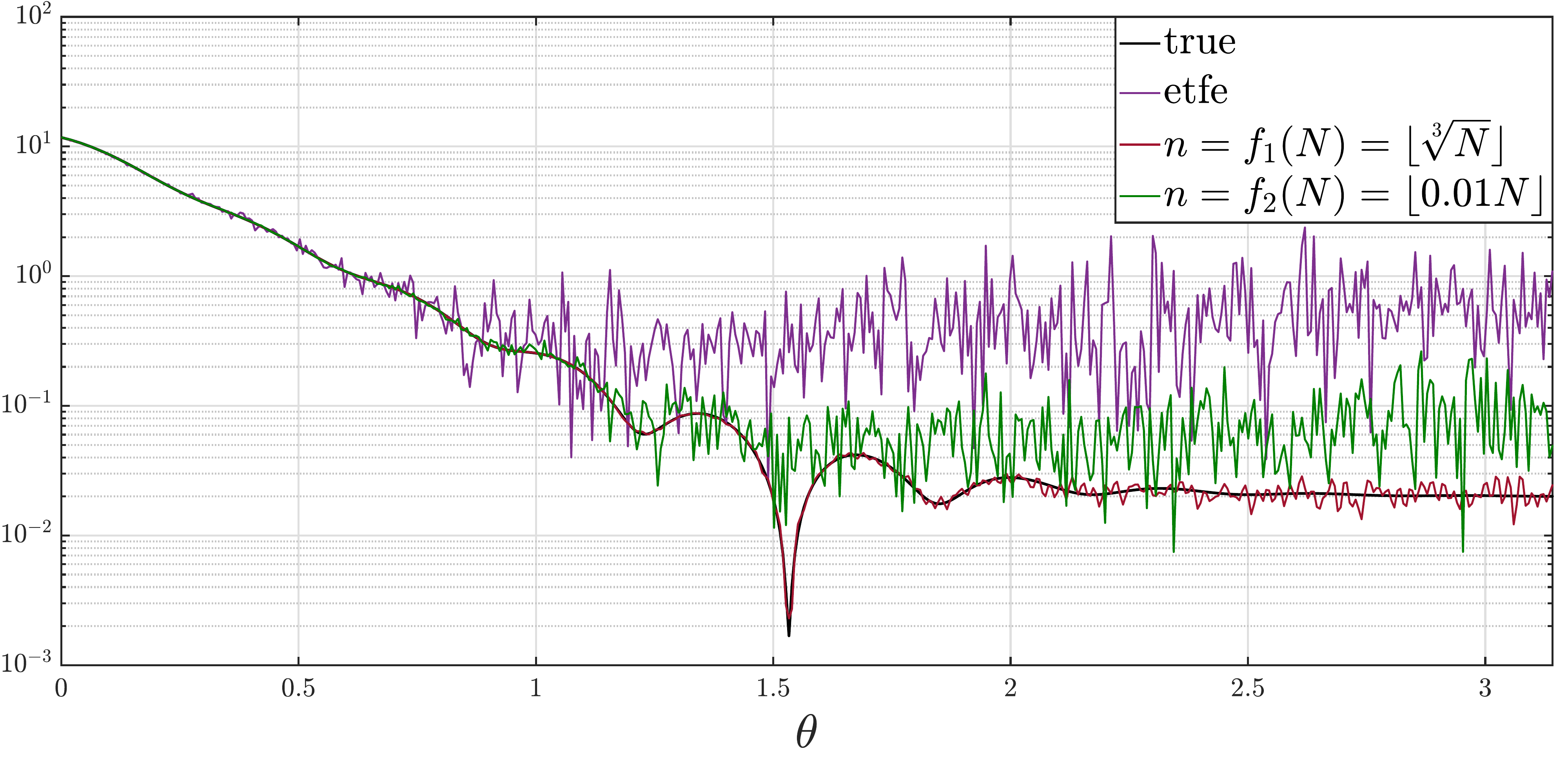}
	\caption{Impulse response estimation problem: absolute value of the true transfer function (thick black line) and one realization (out of the 50) of the transfer function estimators $\hat G_i (e^{i\theta})$, $i = 1,2,3,$ for $N = 2 \cdot 10^7.$}
	\label{fig:etfe_tf}
\end{figure}

\subsection{Learning undirected graphical model} \label{sec:NET}
In many practical situations we have a large number of signals acquired by sensors  and it is crucial to extract information on the relationships between the data.
These situations can be analyzed through graphical interaction models. Given a zero-mean, stationary, Gaussian random process $\yb$ of dimension $m$, a graphical model of the process is an undirected graph with $m$ nodes, one for each component  $y_i$, and an edge connecting the nodes $i$ and $j$ if and only if the components $y_i$ and $y_j$ are conditionally dependent given all the other components\footnote{ We say that $y_i$ and $y_j$ are independent, conditional on all the other $y_k$, if for all $t,s \in \Zbb,$ the random variables $y_i(t)$ and $y_j(s)$ are conditionally independent given the random variables in the closure of the space linearly generated by $\{ y_k(\tau), 1\leq k\leq m, \; k \neq i, \; k \neq j, \; \tau \in \Zbb \}, $ \cite{Avventi_ARMA}.}. 
The conditional independence property has a simple characterization in terms of the spectrum of the process: $y_i$ and $y_j$ are independent, conditional on all the other $y_k$, if and only if  $ \Phi_{ij}^{-1} (e^{i \theta}) = 0 $ for all $\theta$, \cite{ID_DAHLHAUS}. 

In order to estimate the topology of a network we may proceed as follows. 
First, we collect a sample of numerosity $N$ from the model
and we compute an estimate $\hat \Phib $ of the spectrum. Then, the topology of the network is determined by thresholding the estimated inverse spectrum $\hat \Phib^{-1}:$ we assume that an edge connects the nodes $i$ and $j$ if and only if $\Vert \hat \Phi^{-1}_{ij} \Vert\geq c $ where $c$ is the selected threshold. 
To estimate the spectrum one can compute the full (unwindowed) periodogram obtained with (\ref{eq:biasedACS}). Notice however, that the latter  cannot be inverted because it is singular for any $\theta \in [0, 2\pi)$  even in the case that the underlying true spectrum is coercive and the sample size $N$ is large, see e.g. \cite{9029655}. 
This problem is usually even more severe if we take the unbiased periodogram with \eqref{eq:unbiasedACS}; in fact,
in this case the estimate is not even guaranteed positive semi-definite and, if 
the underlying true spectrum is close to singularity for some values of $\theta$, 
this problem remains even when the sample size increases because of  the erratic behaviour of the 
periodogram estimate. 
To overcome the aforementioned issues we propose to consider the $f$-truncated spectral estimator: under the assumption that the underlying true spectrum is coercive, by the  consistency property, it follows that the probability that the $f$-truncated estimator is positive definite pointwise tends to 1 as $N$ approaches infinity.
 
In the following we consider as a specific example the Gaussian process of dimension $m=5$ represented by the graphical model of Figure \ref{fig:network}.
The true inverse power spectral density of the process is 
$ \Phib^{-1} = \Wb^{-*}  \Wb^{-1}, $ where
$$ \Wb^{-1}(z) = \begin{bmatrix}
	\frac{z-0.6}{z+0.5}  & 0 &  0 & \frac{z+0.2}{z-0.3} &  0 \\
	0 &  \frac{z}{z-0.3}  &     \frac{1}{z+0.5} &  0  & 0 \\ 
	0 &  0  & \frac{z+0.7}{z-0.3} &  0   &\frac{z+0.3}{z-0.3}\\ 
	0 &  0  & 0  & \frac{ z+0.1}{z+0.5} & 0 \\
	0 &  0  & 0  & 0  & \frac{z+0.1}{z-0.1} \end{bmatrix}.
$$  
We evaluate the ability to correctly estimate the underlying graphical structure of the following spectral estimators: 
\begin{itemize}
	\item $\hat \Phib_1 (e^{i \theta}):$ the $f_1$-truncated periodogram corresponding to  $f_1(N) = \lfloor \sqrt[3]{N} \rfloor.$ Notice that $f_1$ satisfies the assumption of Theorem \ref{th:cons_1d};
	\item $\hat \Phib_2 (e^{i \theta}):$ $f_2$-truncated periodogram corresponding to $f_2(N) = \lfloor 0.001 N \rfloor.$ Notice that $f_2$ does not satisfy the assumption of Theorem \ref{th:cons_1d};
	\item $\hat \Phib_3 (e^{i \theta}):$ the periodogram smoothed with a rectangular window of fix length $ n = 10.$
\end{itemize}
For all the estimators, the sample covariances $\hat \Rb_k $ are computed by using the formula \eqref{eq:biasedACS}.
We set threshold $c$ to the $ 5 \% $ of the mean value of the norms of the non-zero entries of $ \Phib^{-1},$ that is  $ c = 0.0994,$ and we compare the results of the three approaches. 
In Figure \ref{fig:net_edeges} we analyze the error in the topology selection as a function of the sample size $N$. 
Specifically, we define the topology selection error as the number of misclassified edges in the network and, for each value of $N$, we average this error over $20$ runs of a Monte Carlo simulation. 
Figure \ref{fig:net_Phi} plots the relative estimation error $
\Vert \Phib - \hat \Phib \Vert / \Vert \Phib \Vert,$ where $ \hat \Phib$ is one of the previous spectrum estimates, averaged over the 20 trials for increasing values of $N.$
As expected, both the errors asymptotically go to zero for $\hat \Phib_1,$ which has indeed theoretical guarantees of performance when $N$ tends to infinity; on the other hand, the errors of $\hat \Phib_2 $ and $ \hat \Phib_3,$ after an initial improvement, do not further decrease with $N$.
 \begin{remark}
The fact that the threshold $c$ depends on the true spectral density $\Phib$ appears to be quite artificial.
However, for large values of $N$ nothing changes if we select the threshold with reference
to the estimated  spectral density $ \hat \Phib$.
The real difficulty in the kind of problems like that discussed in this section is to automatically select the threshold in such a way that the asymptotic recovery of the correct topology is guaranteed. 
In our setting this result can be achieved if we assume that $\Phib$ is coercive and that 
we know a lower bound for the rate at which $\|\Rb_k\|$ converges to zero
(for example if $\Phib$ is rational, we know that $\|\Rb_k\|$ converges to zero
faster than $1/k^M$ for any $M>0$).
Indeed, in this case {\luci an} upper bound of the rate of convergence of $\E  \Vert \Deltab (e^{i \theta}) \Vert^2$ can be computed (see the proof of Theorem \ref{th:cons_1d} in Section \ref{sec:multidim}). Then, by means of the continuous mapping Theorem applied to the inverse function and the Markov's inequality, it is possible to select the value $c= g(N) $ of the threshold, where $g$ is a function converging to zero sufficiently slowly, such that the estimated topology converges in probability to the correct topology.
\end{remark} 

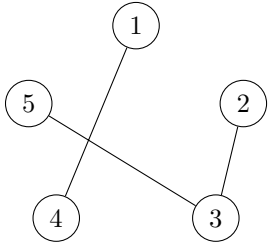
\begin{figure} 
	\centering
	\begin{tikzpicture} 
		\tikzstyle{every node}=[draw,shape=circle];
		\node (v1) at (90 :1.5cm) {$1$};			
		\node (v2) at (90-72: 1.5cm) {$2$};
		\node (v3) at (7*45:1.5cm) {$3$};
		\node (v4) at (5*45:1.5cm) {$4$};
		\node (v5) at (90+72:1.5cm) {$5$};
		\draw (v1) -- (v4)
		(v2) -- (v3)
		(v3) -- (v5);
	\end{tikzpicture}
	\caption{Graphical model of the generative Gaussian random process $\yb$ of dimension $m = 5.$ } \label{fig:network}
\end{figure}
\begin{figure}
	\centering
	\includegraphics[width=\linewidth]{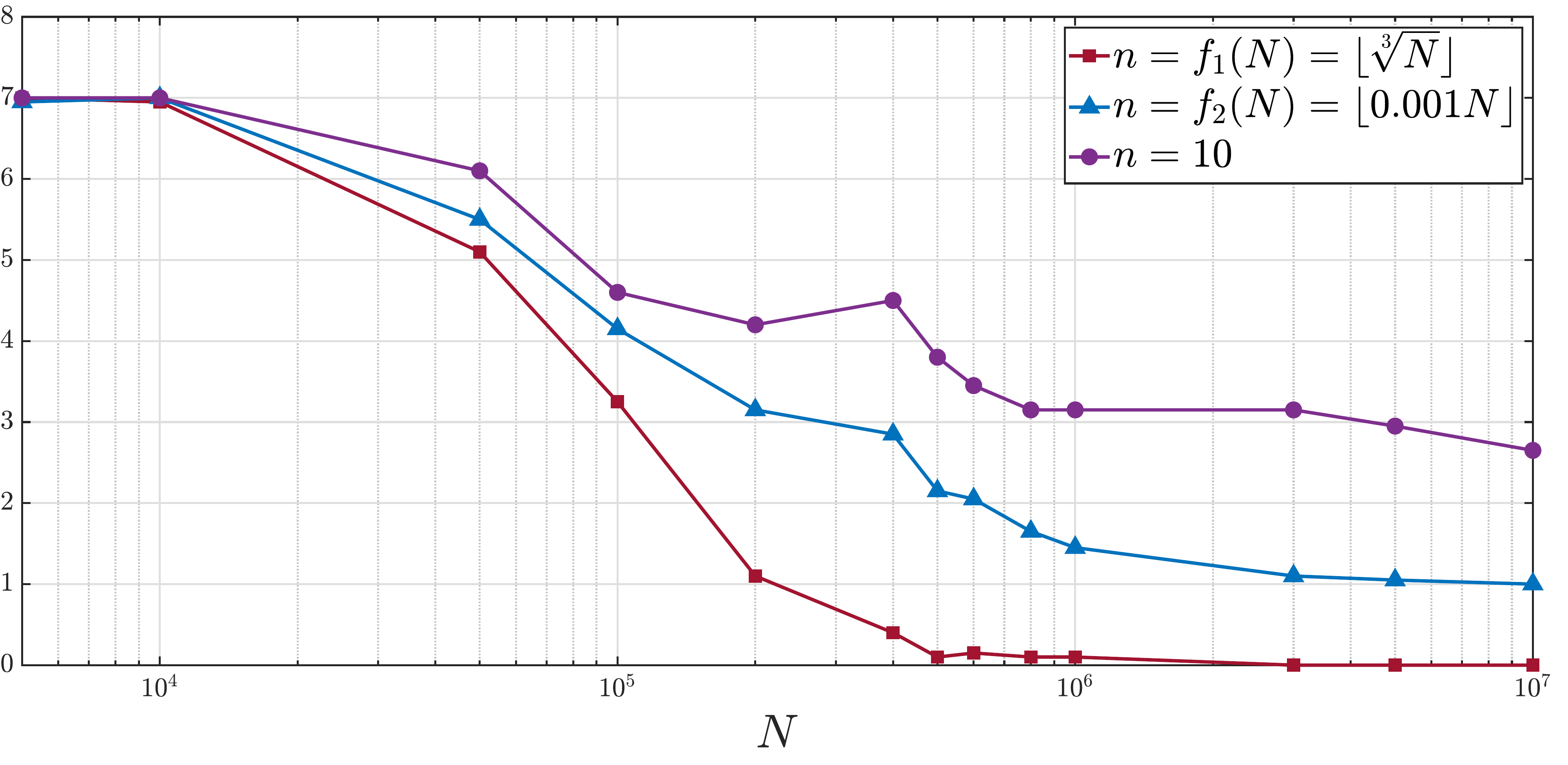}
	\caption{Graphical model: average number of misclassified edges in the estimated graph as a function of the sample size $N.$ Comparison among the different spectral estimators.}
	\label{fig:net_edeges}
\end{figure}
\begin{figure}
	\centering
	\includegraphics[width=\linewidth]{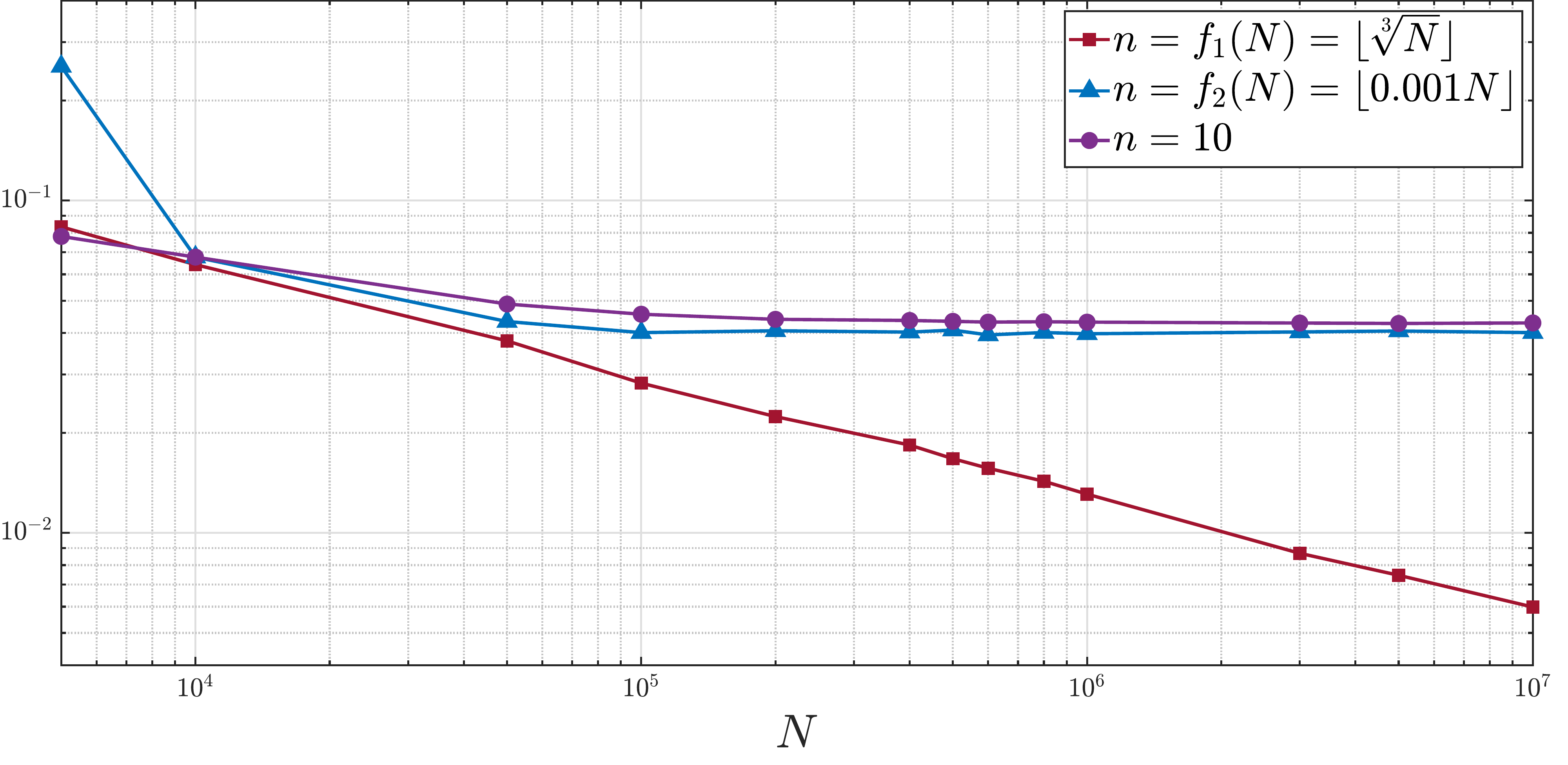}
	\caption{Graphical model: Average relative estimation error $\Vert \Phib - \hat \Phib \Vert / \Vert \Phib \Vert$ as a function of the sample size $N.$ Comparison among the different spectral estimators.}
	\label{fig:net_Phi}
\end{figure}

\section{Multidimensional Random Processes}\label{sec:multidim}
In this section, we extend and prove the consistency of the $f$-truncated periodogram to the class of multidimensional random fields. 
For the sake of clarity, here we first discuss the case in which the $f$-truncated periodogram is built with the unbiased sample covariances and then with the biased covariance estimates.

Consider the second-order stationary  random field
\beq \label{eq:model} 
\yb (\tb) = \sum_{\sigmab \in \Zbb^d} \Mb (\sigmab) \eb (\tb - \sigmab), \qquad \tb = (t_1, \dots, t_d ) \in \Zbb^d,
\eeq 
where the positive integer $d$ is the dimension of the index set, $ \Mb (\cdot) : \Zbb^d \to \Rbb^{m \times p} $ is the impulse response of a BIBO stable filter, \emph{i.e.} 
\beq \label{eq:BIBO}  \sum_{\sigmab \in \Zbb^d} \Vert \Mb (\sigmab ) \Vert  < \infty, \eeq
and $\eb = \{ \eb(\tb),  \tb\in \Zbb^d \}$ is a $p$-dimensional normalized white Gaussian noise. 
Accordingly, $\yb = \{ \yb (\tb), \tb \in \Zbb^d \}$ is a real-valued zero mean random field of dimension $m$. 
The covariance matrix, defined as 
\beq \label{eq:cov} \Rb_\kb := \E \yb (\tb + \kb) \yb (\tb)\tp \eeq
does not depend on $\tb$ by stationary and satisfies the symmetric property $R_\kb = R_{-\kb}\tp. $ 
The spectral density of the random field is the Fourier transform of the matrix field $\{ \Rb_\kb , \kb \in \Zbb^d \},$ \emph{i.e.}
\beq \label{eq:spectrum}
\Phib (e^{i \thetab}) : =  \sum_{\kb \in \Zbb^d} \Rb_\kb
e^{-i \langle \kb,  \thetab \rangle} ,
\eeq
where $\thetab = (\theta_1, \theta_2, \dots,  \theta_d) $ takes values in $ \Tbb^d := [0, 2\pi)^d, $ $e^{i \thetab}$   is a shorthand notation for $ (e^{i \theta_1},  \dots, e^{i \theta_d} ) $ and $  \langle \kb,  \thetab \rangle := k_1 \theta_1 + \dots + k_d \theta_d $ is the usual inner product in $\Rbb^d.$ 

Suppose {we} observe the finite-length realization of the field $ \{  \yb (\tb), 1 \leq t_j \leq N \text{ for } j=1, \dots, d \} $ and we want to estimate the spectrum  from these observations.
The standard unbiased estimate of the covariance sequence is
\beq \label{eq:unbiasedACS_d}
\hat \Rb_\kb := \frac{1}{N_{\kb}} \sum_{\tb \in \Xi_{N,\kb} } \yb (\tb + \kb) \yb (\tb)\tp,
\eeq
where each component of the index $\tb$ satisfies
\beq \label{eq:setXi} \begin{cases}
	1 \leq t_j \leq N - k_j & \text{if } k_j \geq 0 \\
	-k_j +1 \leq t_j \leq N  & \text{if } k_j < 0 \\
\end{cases} \eeq
and hence the set $\Xi_{N,\kb}$ 
$$\Xi_{N,\kb} : = \{ \tb \in \Zbb^d : 
 t_j \text{ satisfies \eqref{eq:setXi} for } j=1,.., d \}
$$
with cardinality $N_\kb  := | \Xi_{N,\kb} | =  (N - |k_1|) ... (N - |k_d|). $ \\
Consider a real function $f$ satisfying the assumptions \eqref{eq:limit_1} and \eqref{eq:limit},
and let $n :=    f(N).$ Then, the $f$-truncated periodogram is defined as
\beq \label{eq:periodogram}
\hat \Phib (e^{i \thetab}) : =  \sum_{\kb \in \Lambda_n} \hat \Rb_\kb e^{-i \langle \kb,  \thetab \rangle}
\eeq
with  $ \Lambda_n := \{ \kb \in \Zbb^d : |k_j| \leq n \text{ for } j=1, \dots, d\} $.
The following theorem guarantees the mean-square consistency of $\hat \Phib:$

\begin{theorem}\label{th:cons_unbias}
Given any random field  $\yb$ of the form \eqref{eq:model}, the $f$-truncated periodogram \eqref{eq:periodogram} with $n$ defined by \eqref{eq:n}-\eqref{eq:limit_1}-\eqref{eq:limit} and $\hat \Rb_k$ estimated through \eqref{eq:unbiasedACS_d} is a uniformly mean-square consistent estimator of the spectral density, that is
$$ 
\lim_{N \to \infty} \E  \Vert \Deltab (e^{i {\luci \thetab} }) \Vert^2  = 0 \quad \quad \text{uniformly over } {\luci \Tbb^d}, 
$$
where $ \Deltab(e^{i {\luci \thetab}}) := \Phib(e^{i {\luci \thetab}}) - \hat \Phib (e^{i {\luci \thetab}}). $	
\end{theorem}

\begin{proof}
	By plugging \eqref{eq:model} into \eqref{eq:cov}, it is easy to get
	$$
	\Rb_\kb = \sum_{\sigmab \in \Zbb^d}  \Mb (\sigmab + \kb )  \Mb (\sigmab)\tp
	$$
	and therefore,
	$$ 
			\Vert \Rb_\kb \Vert^2 = \sum_{\sigmab_a, \sigmab_b}  \Tr \big( \Mb (\sigmab_a ) \Mb (\sigmab_a + \kb )\tp  \Mb (\sigmab_b + \kb) \Mb (\sigmab_b)\tp \big) .
	$$
	Notice that the BIBO stability assumption \eqref{eq:BIBO} implies
	\begin{align} \label{eq:ell1}
		\sum_{\kb \in \Zbb^d}& \Vert \Rb_\kb  \Vert \leq  \sum_{\kb, \sigmab \in \Zbb^d} \Vert {\luci \Mb} (\sigmab + \kb )\Vert \; \Vert {\luci \Mb} (\sigmab ) \Vert \nonumber \\
		&= \sum_{\sigmab \in \Zbb^d} \Vert {\luci \Mb} (\sigmab ) \Vert  \sum_{\kb \in \Zbb^d} \Vert {\luci \Mb} (\sigmab + \kb )\Vert  <  \infty  . 
	\end{align}	
	From \eqref{eq:model} and \eqref{eq:unbiasedACS_d}, we immediately obtain
	\begin{multline*}
		\hat \Rb_\kb = \frac{1}{N_{\kb}} \sum_{\tb \in \Xi_{N,\kb} }  \sum_{\sigmab_a, \sigmab_b \in \Zbb^d}  \Mb (\sigmab_a) \eb(\tb + \kb - \sigmab_a) \\ \times \eb(\tb - \sigmab_b)\tp \Mb (\sigmab_b) \tp ,
	\end{multline*}
	and clearly $ \E \hat \Rb_\kb =  \Rb_\kb$ for any $\kb$ and $N$.
	Moreover,
	\begin{align*} 
		 \E  \Vert   \hat \Rb_\kb  \Vert^2  & = \E \Bigg[  \Tr \Big(   \frac{1}{\luci N_{\kb}^2 } \sum_{\tb_a, \tb_b} \sum_{\substack{\sigmab_a,\sigmab_b, \\ \sigmab_c,\sigmab_d}} 
		\Mb(\sigmab_b) \eb(\tb_a - \sigmab_b) \\ 
		& \times \eb(\tb_a + \kb -\sigmab_a)\tp  \Mb(\sigmab_a) \tp  \Mb (\sigmab_c) \\
		& \times  \eb(\tb_b + \kb - \sigmab_c) \eb(\tb_b - \sigmab_d)\tp \Mb (\sigmab_d) \tp \Big)  \Bigg] \\
		&=  \frac{1}{  N_{\kb}^2 } \sum_{\tb_a, \tb_b}  \sum_{\substack{\sigmab_a,\sigmab_b, \\ \sigmab_c,\sigmab_d}}   \sum_{q,z =1}^{m} \sum_{i,j,h,l=1}^{p} M_{qi}(\sigmab_b)M_{zj}(\sigmab_a)  \\  
		& \times  M_{zh}(\sigmab_c)  M_{ql}(\sigmab_d)\E \Big[ e_i (\tb_a - \sigmab_b) \\ 
		&  \times    e_j (\tb_a + \kb - \sigmab_a ) e_h (\tb_b + \kb - \sigmab_c ) e_l (\tb_b - \sigmab_d) \Big].
	\end{align*}
	Now, in order to evaluate the expectation in the above formula, by recalling that $\eb$ is a white Gaussian noise, we distinguish the following cases: 
	if 
	\begin{equation*}
		\begin{cases}
			i=j, \quad h=l  \\ \tb_a - \sigmab_b = \tb_a + \kb - \sigmab_a  \\ \tb_b + \kb - \sigmab_c = \tb_b - \sigmab_d  \\ \tb_a - \sigmab_b  \neq  \tb_b - \sigmab_d 
		\end{cases}
	\end{equation*}
	or 
	\begin{equation*}
		\begin{cases}
			i=h, \quad j = l \\  \tb_a - \sigmab_b = \tb_b + \kb - \sigmab_c \\ \tb_a + \kb - \sigmab_a = \tb_b - \sigmab_d \\  \tb_a - \sigmab_b  \neq  \tb_b - \sigmab_d
		\end{cases}
	\end{equation*}
	or 
	\begin{equation*}
		\begin{cases}
			i=l, \quad j = h \\  \tb_a - \sigmab_b = \tb_b - \sigmab_d  \\ \tb_a + \kb - \sigmab_a = \tb_b + \kb - \sigmab_c \\   \tb_a - \sigmab_b \neq \tb_a + \kb - \sigmab_a,
		\end{cases}
	\end{equation*}
	then the expected value is 1.
	Moreover, if $\tb_a - \sigmab_b = \tb_a + \kb - \sigmab_a = \tb_b + \kb - \sigmab_c  = \tb_b - \sigmab_d $, the expectation is 1 if $i=j$, $h = l,$ $i\neq h$, or if $i=h$, $j = l,$ $i\neq j$, or if $i=l$, $j = h,$  $i\neq j$, while it is equal to 3 if $ i=j=h=l. $
	In all the other cases the expected value is zero. 
	In view of these observations, we have:
	\begin{align*}
		 \E  \Vert   \hat \Rb_\kb & \Vert^2  =  \frac{1}{  N_{\kb}^2 } \Big( \sum_{\tb_a, \tb_b}  \sum_{\sigmab_b, \sigmab_d}  \sum_{q,z  =1}^{m} \sum_{i,h=1}^{p}  M_{qi}(\sigmab_b)   \\
		& \times M_{zi}(\sigmab_b + \kb)  M_{zh}(\sigmab_d + \kb)  M_{q}(\sigmab_d) \\  
		& +    \sum_{\tb_a}  \sum_{\sigmab_a, \sigmab_b, \sigmab_d}  \sum_{q,z  =1}^{m} \sum_{i,j=1}^{p}  M_{qi}(\sigmab_b) M_{zj}(\sigmab_a) \\
		& \times M_{zi}(\sigmab_b + \sigmab_d  - \sigmab_a + 2\kb)    M_{qj}(\sigmab_d)  \\
		& +  \sum_{\tb_a}  \sum_{\sigmab_a, \sigmab_b, \sigmab_c}  \sum_{q,z  =1}^{m} \sum_{i,j=1}^{p}  M_{qi}(\sigmab_b) \\
		& \times M_{zj}(\sigmab_a) M_{zj}(\sigmab_c) M_{qi}(\sigmab_b + \sigmab_c - \sigmab_a)  \Big) \\
		&\leq  \Vert  \Rb_\kb  \Vert^2 + \frac{C_1}{N_{\kb}}  + \frac{C_2}{N_{\kb}} = \Vert  \Rb_\kb  \Vert^2 + \frac{C}{N_{\kb}}
		\numberthis \label{eq:rhat_square}
	\end{align*}
	where the inequality follows from the BIBO stability assumption \eqref{eq:BIBO} and $C:= C_1 + C_2$ is a constant independent of $\kb$ and $N$.	
	Consequently,  
	\begin{align} \label{eq:cov_error}
		\E  \Vert  \Rb_\kb - \hat\Rb_\kb  \Vert^2 &=  \E \Vert \hat \Rb_\kb  \Vert^2 - \Vert \Rb_\kb  \Vert^2  \leq \frac{C}{N_{\kb}} .
	\end{align} 
	Now, define 
	\beq \label{eq:S1} \Sb_1 (e^{i \thetab}) := \sum_{\kb \in \Lambda_n} (\Rb_\kb -  \hat \Rb_\kb ) e^{-i  \langle \thetab, \kb \rangle} \eeq 
	and 
	\beq \label{eq:S2}\Sb_2 (e^{i \thetab}) : = \sum_{\kb \in \Zbb^d \setminus \Lambda_n}  \Rb_\kb  e^{-i  \langle \thetab, \kb \rangle},\eeq 
	so that $\Deltab = \Sb_1 + \Sb_2.$
	From \eqref{eq:limit_1} and \eqref{eq:ell1} it follows that $$\lim_{N \to \infty} \Vert  \Sb_2 (e^{i \thetab}) \Vert = 0.$$
	Moreover, $ \E \Vert  \Sb_1 (e^{i \thetab}) \Vert = 0$ and  
	\begin{align*}
		\E   \Vert  \Sb_1 &(e^{i \thetab}) \Vert^2 \leq  \sum_{\kb_a, \kb_b } \sqrt{ \E  \Vert \Rb_{\kb_a} - \hat \Rb_{\kb_a} \Vert^2   \E \Vert \Rb_{\kb_b} - \hat \Rb_{\kb_b} \Vert^2   } \\
		& \leq  \sum_{\kb_a, \kb_b \in \Lambda_n}  \sqrt { \frac{C}{N_{\kb_a}}  \frac{C}{N_{\kb_b}} }  \\
		& \leq  \sum_{\kb_a, \kb_b \in \Lambda_n}  \frac{C}{(N - \max \{k_{a_1}, ..., k_{a_d}, k_{b_1}, ..., k_{b_d}\})^d } \\
		& \leq (2n+1)^{2d}\frac{C}{(N-n)^d}
	\end{align*}
	where the first inequality is a consequence of the Cauchy–Schwarz inequality, and second inequality comes from \eqref{eq:cov_error}. 
	Notice that under the assumption \eqref{eq:limit}, the right hand side of the final inequality tends to zero as $N$ grows to infinity.
	This suffices to conclude the proof. Indeed,  for each $\thetab$
	\begin{align*}
		\lim_{N \to \infty} \E  \Vert \Deltab (e^{i \thetab}) \Vert^2  &=\lim_{N \to \infty}  \{ \E \Vert \Sb_1 (e^{i \thetab}) \Vert^2 +  \Vert \Sb_2 (e^{i \thetab}) \Vert^2  \Big\} \\& = 0.  \qquad \Box  	
	\end{align*}
\end{proof}

\begin{remark}
It is interesting to observe that the weaker assumption  
	\beq \label{eq:l2sys} \sum_{\sigma \in \Zbb^d} \Vert \Mb (\sigma) \Vert^2 < \infty 	\eeq 
is sufficient to guarantee that the process in question is second-order stationary with finite variance matrix (see \cite[Ch.3]{priestley1982spectral}). \\
If we give up the BIBO stability assumption and we only assume \eqref{eq:l2sys}, the inequalities \eqref{eq:rhat_square} and \eqref{eq:cov_error} still hold. However we have that $ \Rb_{\kb} \in \ell_2 (\mathbb{Z}^d)$ so that the spectral density $\Phib$ is defined in the $L_2$ sense .  In other words  $ \sum_{\kb \in \Zbb^d}  \Rb_\kb  e^{-i  \langle \thetab, \kb \rangle} $ converges in the $L_2$-norm (and not necessarily pointwise) to  a function in $L_2(\mathbb{T}^d)$  by the Riesz-Fischer Theorem \cite[pag.91-92]{rudinreal}. 
Consequently,
$$
\Sb_2 (e^{i \thetab}) = \sum_{\kb \in \Zbb^d}  \Rb_\kb  e^{-i  \langle \thetab, \kb \rangle} - \sum_{\kb \in \Lambda_n}  \Rb_\kb  e^{-i  \langle \thetab, \kb \rangle}
$$
converges in the $L_2$-norm to zero as $N$ (and thus $n = f(N)$) grows to infinity.
In such a case the $f$-truncated periodogram  \eqref{eq:periodogram} with $n$ defined by \eqref{eq:n}-\eqref{eq:limit_1}-\eqref{eq:limit} and $\hat \Rb_k$ estimated through \eqref{eq:unbiasedACS_d} is  $L_2$-consistent, i.e. 
$$ 
\lim_{N \to \infty}  \int_{\Tbb^d} \E \Vert \Deltab (e^{i \thetab}) \Vert^2  d \thetab = 0 .
$$
In other words, our result remains valid provided that we consider consistency in the $L_2$
sense and not pointwise in $\thetab$.
\end{remark}

Next, we prove that the same consistency result of Thereom \ref{th:cons_unbias} holds when, in place of \eqref{eq:unbiasedACS_d}, we consider the biased covariance estimates
\beq \label{eq:biasedACS_d}
\hat \Rb_\kb := \frac{1}{N_{\zerob}} \sum_{\tb \in \Xi_{N,\kb} } \yb (\tb + \kb) \yb (\tb)\tp,
\eeq
with $N_{\zerob} := N^d,$ in the  truncated periodogram \eqref{eq:periodogram}.

\begin{theorem}\label{th:cons_bias}
Given any random field  $\yb$ of the form \eqref{eq:model}, the $f$-truncated periodogram \eqref{eq:periodogram} with $n$ defined by \eqref{eq:n}-\eqref{eq:limit_1}-\eqref{eq:limit} and $\hat \Rb_k$ estimated through \eqref{eq:biasedACS_d} is a uniformly mean-square consistent estimator of the spectral density.
\end{theorem}

\begin{proof}
The proof is similar to the proof of Theorem \ref{th:cons_unbias}. 
The main difference is that 
\begin{multline*}
	\hat \Rb_\kb := \frac{1}{N_{\zerob}} \sum_{\tb \in \Xi_{N,\kb} } \sum_{\sigmab_a,\sigmab_b \in \Zbb^d}  \Mb (\sigmab_a) \eb(\tb + \kb - \sigmab_a) \times \\ \eb(\tb - \sigmab_b)\tp \Mb (\sigmab_b) \tp 
\end{multline*}
and then
$$ \E \hat \Rb_\kb = \frac{N_\kb}{N_{\zerob}} \Rb_\kb.$$
Repeating the same reasoning as in the proof of Theorem \ref{th:cons_unbias}, it is not difficult to see that inequality \eqref{eq:rhat_square} becomes
\begin{align*}
 \E  \Vert   \hat \Rb_\kb  \Vert^2  & \leq \left(\frac{N_{\kb}}{N_{\zerob}} \right)^2  \Vert  \Rb_\kb  \Vert^2 + \frac{N_{\kb} C_1}{N_{\zerob} ^2} + \frac{N_{\kb} C_2}{N_{\zerob}^2} \\ 
&\leq \left(\frac{N_{\kb}}{N_{\zerob}} \right)^2  \Vert  \Rb_\kb  \Vert^2 + \frac{C}{N_{\zerob}} . 
\end{align*}
Consequently,  
\begin{align*} 
\E  \Vert  \Rb_\kb - \hat\Rb_\kb  \Vert^2 &=  \left(1 - 2\frac{N_\kb}{N_{\zerob}} \right) \Vert \Rb_\kb\Vert^2 + \E \Vert \hat \Rb_\kb  \Vert^2 
 \\& \leq \left( 1 - \frac{N_\kb}{N_{\zerob}}  \right)^2 \Vert \Rb_\kb\Vert^2 + \frac{C}{N_{\zerob}} .
 \numberthis \label{eq:cov_error_bias}
\end{align*} 
Let $\Sb_1$ and $\Sb_2$ be defined by \eqref{eq:S1} and \eqref{eq:S2}, respectively and $\Deltab = \Sb_1 + \Sb_2.$  
We have already observed in Theorem \ref{th:cons_unbias} that $\lim_{N \to \infty} \Vert  \Sb_2 (e^{i \thetab}) \Vert = 0.$
Next, we prove that also $\Vert  \Sb_1 (e^{i \thetab}) \Vert$ converges to zero {\luci in the mean-square sense}. To this end, we notice that 
\begin{align*}
\E   \Vert  \Sb_1 (e^{i \thetab})  \Vert^2 &\leq  \sum_{\kb_a, \kb_b} \sqrt{ \E  \Vert \Rb_{\kb_a} - \hat \Rb_{\kb_a} \Vert^2   \E \Vert \Rb_{\kb_b} - \hat \Rb_{\kb_b} \Vert^2   } \\
& \leq  \sum_{\kb_a \in \Lambda_n} \sqrt {  \big[( 1 - \frac{N_{\kb_a}}{N_{\zerob}} )^2\Vert \Rb_{\kb_a} \Vert^2   +  \frac{C}{N_{\zerob}} \big]} \\ 
&   \times  \sum_{\kb_b \in \Lambda_n} \sqrt { \big[( 1 - \frac{N_{\kb_b}}{N_{\zerob}} )^2\Vert \Rb_{\kb_b} \Vert^2   +  \frac{C}{N_{\zerob}} \big] }  \\
& \leq  \Big( \sum_{\kb \in \Lambda_n} ( 1 - \frac{N_{\kb}}{N_{\zerob}} )\Vert \Rb_{\kb}\Vert \Big) ^2 +
2 \sqrt{\frac{C}{N_{\zerob}}} \\
& \times \sum_{\kb \in \Lambda_n}(1 - \frac{N_{\kb}}{N_{\zerob}})\Vert \Rb_{\kb}\Vert + (2n+1)^{2d}\frac{C}{N_{\zerob}},
\end{align*}
where the first inequality is a consequence of the Cauchy–Schwarz inequality, and second inequality comes from \eqref{eq:cov_error_bias}. 
The right side of the last inequality is given by the sum of three terms, of which the last obviously tends to zero as $N$ grows to infinity because of \eqref{eq:limit}. As regards the other two terms, notice that 
$$ 
\sum_{\kb \in \Lambda_n} ( 1 - \frac{N_{\kb}}{N_{\zerob}} )\Vert \Rb_{\kb}\Vert \leq \sum_{j=0}^{d} \binom{d}{j}\left(\frac{n}{N}\right)^j  \sum_{\kb \in \Lambda_n} \Vert \Rb_{\kb}\Vert ,
$$
where $\binom{d}{j} $ is the binomial coefficient.
The right-hand side goes to zero if condition \eqref{eq:limit} is satisfied since the infinite sum  $\sum_{\kb \in \Zbb^d} \Vert \Rb_{\kb}\Vert$ is assumed to be convergent.
Hence $\lim_{N \to \infty} \E   \Vert  \Sb_1 (e^{i \thetab}) \Vert^2 = 0$ and, as a consequence,  also $\lim_{N \to \infty} \E   \Vert  \Sb_1 (e^{i \thetab}) \Vert = 0.$
This concludes the proof. Indeed, for each $\thetab$
\begin{align*}
	\lim_{N \to \infty} \E  \Vert \Deltab (e^{i \thetab}) \Vert^2  &=\lim_{N \to \infty}  \Big\{ \E \Vert \Sb_1 (e^{i \thetab}) \Vert^2 +  \Vert \Sb_2 (e^{i \thetab}) \Vert^2   \\&   + 2  \Tr \E [\Sb_1 (e^{i \thetab})] \Sb_2 (e^{i \thetab}) \tp \Big\}  = 0.  \qquad 
	\Box  	
\end{align*}
\end{proof}	
So far we have assumed that $\yb$ is a real-valued signal. However, the result can be easily generalized to the complex-valued case by replacing the transpose operator $(\cdot)\tp$ with the complex-conjugate operator $(\cdot)^*$ in the previous definitions and computations:
\begin{theorem}
	Let $\yb(\tb)$ be a second-order stationary complex $m$-valued random field defined over $ \Zbb^d.$ Suppose that the signal is obtained by linearly filtering the $p$-dimensional circular white noise 	sequence	of unit variance $\eb = \{ \eb(\tb),  \tb\in \Zbb^d \}$
		\footnote{A sequence $\{\eb(\tb), \; t\in\mathbb Z^d \}$ is called \emph{circular white noise} if it satisfies 
		$\E \eb(\tb)\eb^*(\ssb) = \sigma^2 \delta_{\tb,\ssb} $ and $ \E{\eb(\tb)\eb(\ssb)\tp}
		= 0$ for all $\tb$ and $\ssb$. Note that $\sigma^2 =	\E \eb(\tb)\eb^*(\tb)$ is the variance of the signal \cite[p.32]{stoica2005spectral}.} 
	:
	$$ 
	\yb (\tb) = \sum_{\sigmab \in \Zbb^d} \Mb (\sigmab) \eb (\tb - \sigmab), \qquad \tb = (t_1, \dots, t_d ) \in \Zbb^d.
	$$
	Here, the impulse response of the filter $ \Mb (\cdot) : \Zbb^d \to \Cbb^{m \times p} $ satisfies the assumption \eqref{eq:BIBO}. 
	Let $\hat \Phib $ be the $f$-truncated periodogram \eqref{eq:periodogram} where 
	$$
	\hat \Rb_\kb := \frac{1}{N_\kb} \sum_{\tb \in \Xi_{N,\kb} } \yb (\tb + \kb) \yb (\tb)^*.
	$$
	or
	\beq\label{eq:unbACS_complex}
	\hat \Rb_\kb := \frac{1}{N_{\zerob}} \sum_{\tb \in \Xi_{N,\kb} } \yb (\tb + \kb) \yb (\tb)^*
	\eeq
	and $n$ is given by \eqref{eq:n}-\eqref{eq:limit_1}-\eqref{eq:limit}. Then,  $\hat \Phib $  is a uniformly mean-square consistent estimator of the spectral density $\Phib$ of the field. 
\end{theorem}

In the following we give a concrete problem that can be solved using the proposed estimator.

\subsection{Automotive {\luci target} parameter estimation} \label{sec:RAD}
In this {\luci subsection}, we consider the problem of target reconstruction in an automotive radar systems. \\
Radar technology is currently used in many applications of advanced driver assistance systems (ADASs) to estimate locations and velocities of surrounding targets.
State-of-the-art radar sensors use the chirp sequence modulation principle and an uniform array of receive antennas (ULA) to independently measure the range $r$, the relative velocity $v$ and the angle $\alpha$ of multiple targets in the field of view. 
Assume that only one target is present in the radar coverage and a ULA of receive antennas is used for the measurement as shown in Figure \ref{fig:radar}. 
According to \cite{ZHU2021}, the (scalar) measurements of a ULA in a coherent processing interval (CPI) can be approximated through the rational model:
\beq  \label{eq:radar_AR} 
y(\tb) = \frac{1}{1 - \langle \alphab,\zb^{-1} \rangle} v(\tb) +  w(\tb).
\eeq 
The index $\tb$ takes value in the set $ \{ \tb = (t_1,t_2,t_3) \in \Zbb^3 : 0 \leq t_j \leq N_j -1, j=1,2,3  \},$ where $N_1, N_2$ and $N_3$ are the number of range samples, pulses and antennas, respectively, and $\Nb := [N_1, N_2, N_3]$ defines the size of data array; $\alpha_j = \rho_j e^{i \omega_j}$ ($j=1,2,3$)  with $\rho_j $ such that the sum $\rho_1 + \rho_2 + \rho_3$ is close to one; $\langle \alphab,\zb^{-1} \rangle := \alpha_1 z_1^{-1} + \alpha_2 z_2^{-1} + \alpha_3 z_3^{-1};$ $v(\tb)$ is white noise with unit variance;  the process $w$ is circular complex white noise with covariance $\lambda^2$. 
The real vector $\omegab = [\omega_1,\omega_2, \omega_3] \in [0, 2\pi)^3$ contains the three unknown angular target frequencies from which we can readily recover the target range, relative velocity and azimuth angle (see \cite{Engels2014}).
\begin{figure}
	\centering
	\includegraphics[width=\linewidth]{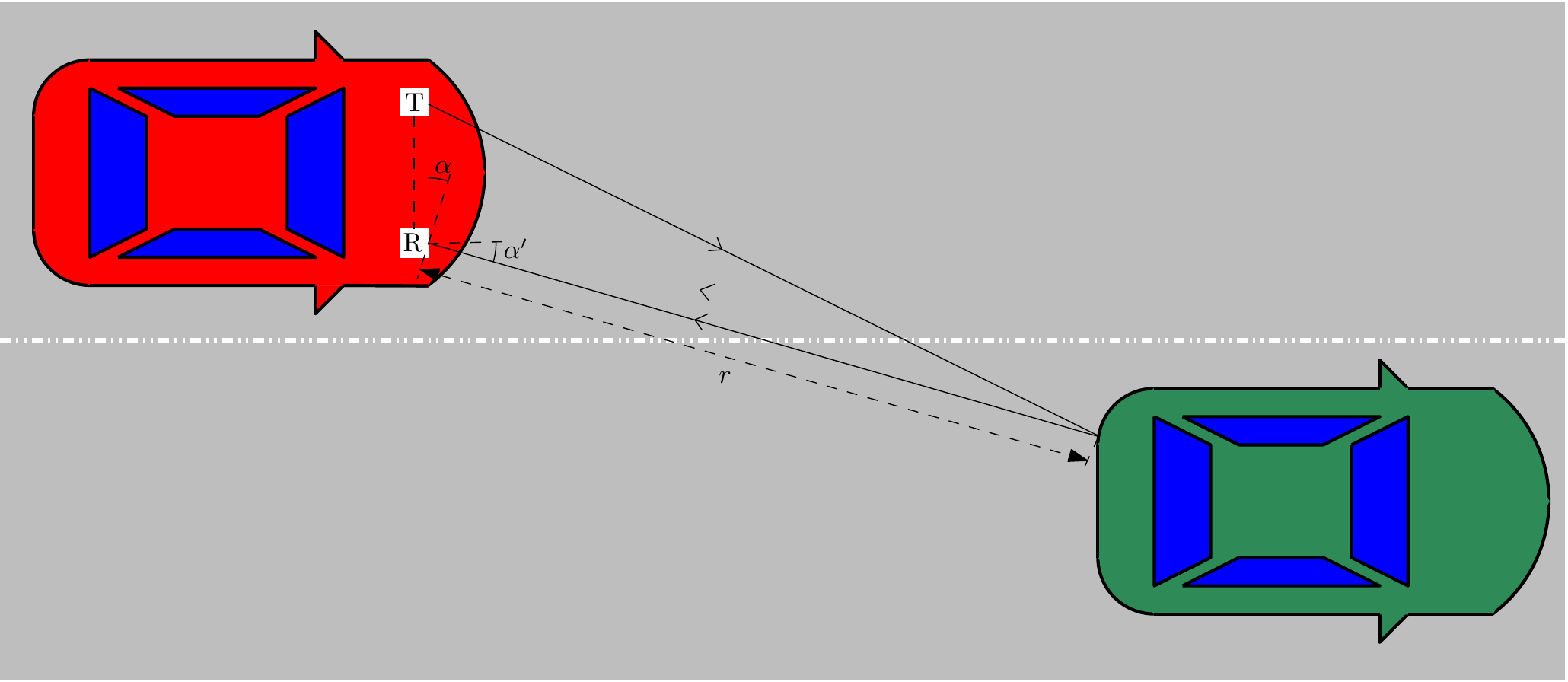}
	\caption{{\em Target parameter estimation set-up.} The automotive radar system is installed in the red car; the green car is the target. $T$ is the transmitter; $R$ is the ULA receiver, $r$ is range of the target; $\alpha$ is the azimuth angle.
	Under the \emph{far field assumption} which is common in this kind of setup $\alpha \approxeq \alpha'.$ }
	\label{fig:radar}
\end{figure}  
The true spectrum $\Phi$ of the random field \eqref{eq:radar_AR} is given by
$$
\Phi(e^{i\thetab}) = \frac{1}{ |1 -\langle \alphab, e^{-i \thetab } \rangle|^2 } + \lambda^2
$$
and it clearly shows a peak at the vector $\omegab.$ 

We give the results of a Monte Carlo simulation composed by $20$ trials for each value of the data array $\Nb.$ 
In each trial, every component of the target frequency vector $\omegab$ is drawn from the uniform distribution in $[0,2\pi);$
the pole moduli are fixed as $\rho_1 = \rho_2 = \rho_3 = 0.3;$ the measurement noise $w$ is zero-mean Gaussian with variance $\lambda^2= 2$ . 
To keep it simple we assume that $N_1 = N_2 = N_3 = N.$ After choosing the size $N$, a realization of the process $y$ is generated according to \eqref{eq:radar_AR} considering zero boundary conditions.  
Then, the full (unwindowed) periodogram and the periodogram \eqref{eq:periodogram} smoothed with a rectangular window of length $n =  f(N) = \lfloor\sqrt[3]{N} \rfloor$ are computed; the biased estimates \eqref{eq:unbACS_complex} of the covariance lags are considered.
Once an estimate $ \hat \Phi$ of the spectrum is available, it is possible to estimate the peak location $\hat \omegab$ as
$$
\hat \omegab = \arg\max_{\theta \in \Tbb^3} \hat \Phi (\theta).
$$
We are interested in how well the true spectrum is approximated and the capability of these estimators to correctly locate the peak. 
Hence, we measure the average (among $20$ trials of a Monte Carlo simulation) of the relative estimation error of the spectrum approximation $\Vert \hat \Phi - \Phi \Vert / \Vert \Phi\Vert,$  and the average of the error of peak finding  $\Vert \hat \omegab - \omegab \Vert.$ These quantities are depicted in Figure \ref{fig:radar_phi} and \ref{fig:radar_theta} for increasing values of $N$. As one can see, only the performance of the truncated periodogram improves with the sample size.  

We also consider one specific example in which the data size is $\Nb = [1000, 1000, 1000]$ and the true frequency vector for the data generation is $\omegab = [2.58, 1.07, 2.88].$ Figure \ref{fig:radar_sec1}, \ref{fig:radar_sec2} and \ref{fig:radar_sec3} display the sections of the two estimated spectra in the first, second and third dimension, respectively, in correspondence of the true peak value $\omegab.$ The figures confirms that the $f$-truncated periodogram is able to correctly locate the peak and that it clearly outperforms the full periodogram. It is worth noting that the periodogram's performance is not good because it is not able to approximate well the shape of the peak which is a consequence of the fact that its values at adjoining frequencies are asymptotically uncorrelated.

\begin{figure}
	\centering
	\includegraphics[width=\linewidth]{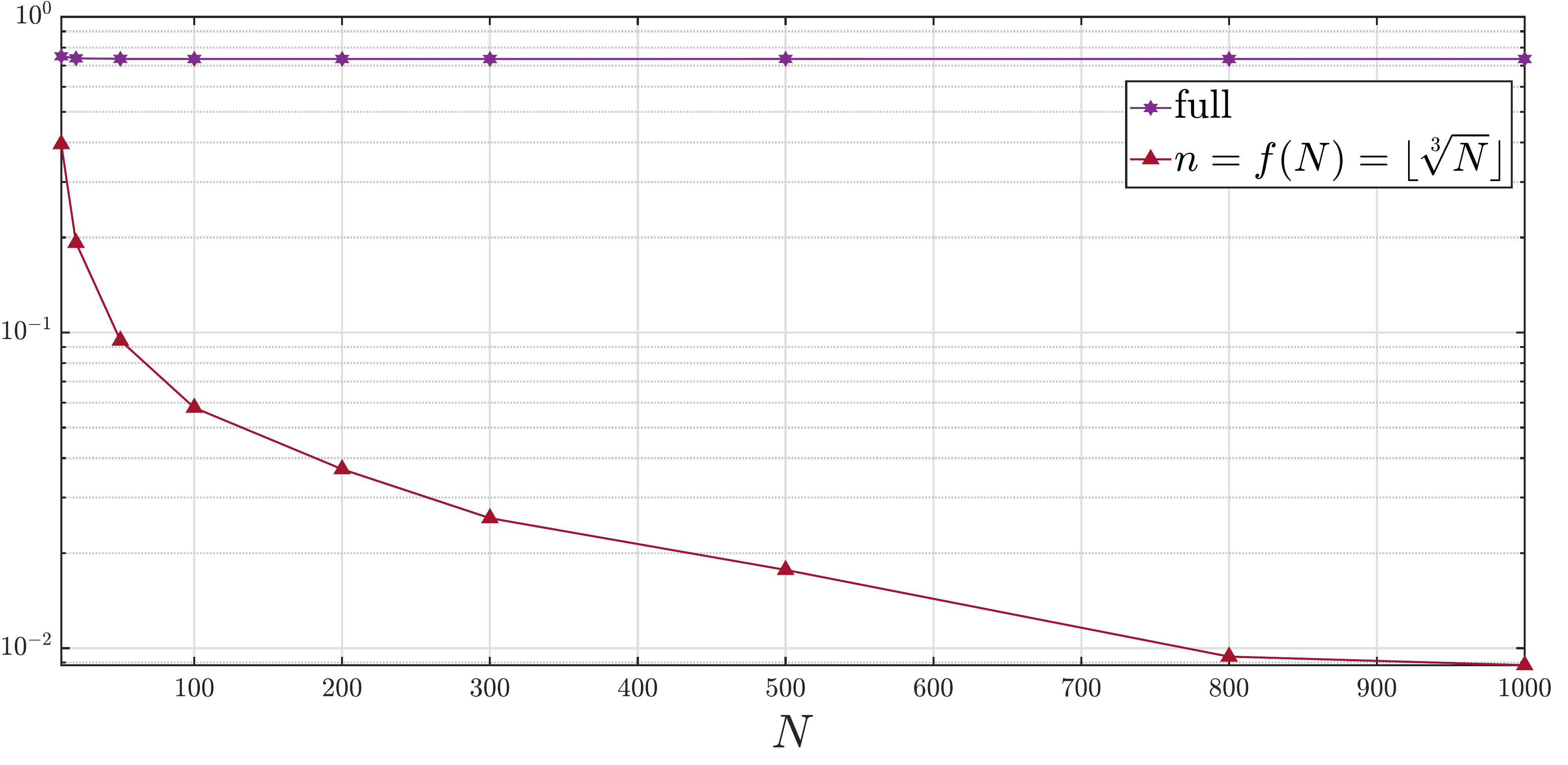}
	\caption{{ Automotive Target parameter estimation:} average relative estimation error $\Vert \Phib - \hat \Phib \Vert / \Vert \Phib \Vert$ as a function of the sample size $N.$ Comparison between the two spectral estimators.}
	\label{fig:radar_phi}
\end{figure}
\begin{figure}
	\centering
	\includegraphics[width=\linewidth]{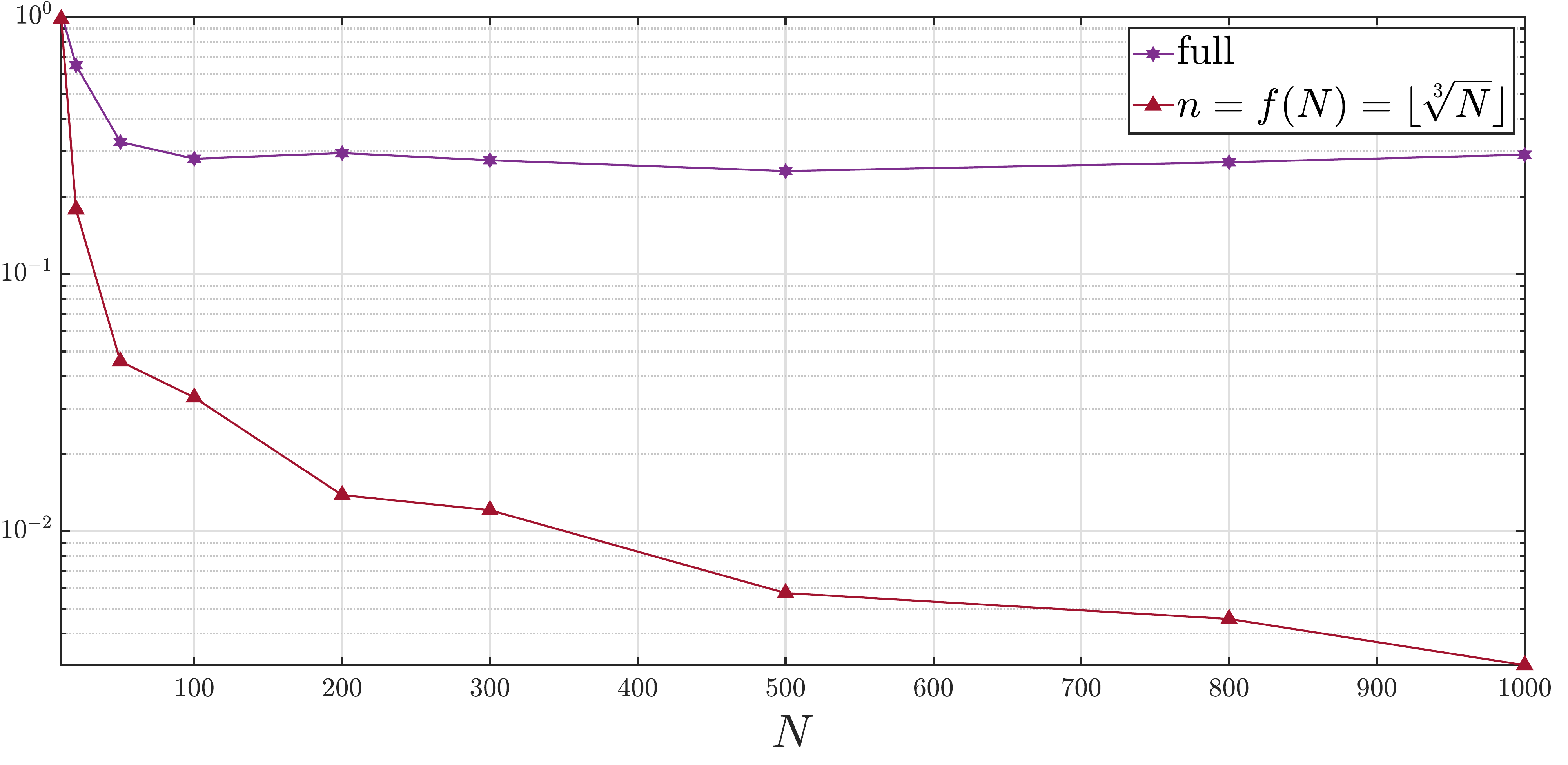}
	\caption{{Automotive Target parameter estimation:} average frequency estimation error $\Vert \hat \omegab - \omegab \Vert$ as a function of the sample size $N.$ Comparison between the two spectral estimators. }
	\label{fig:radar_theta}
\end{figure}
\begin{figure}
	\centering
	\includegraphics[width=\linewidth]{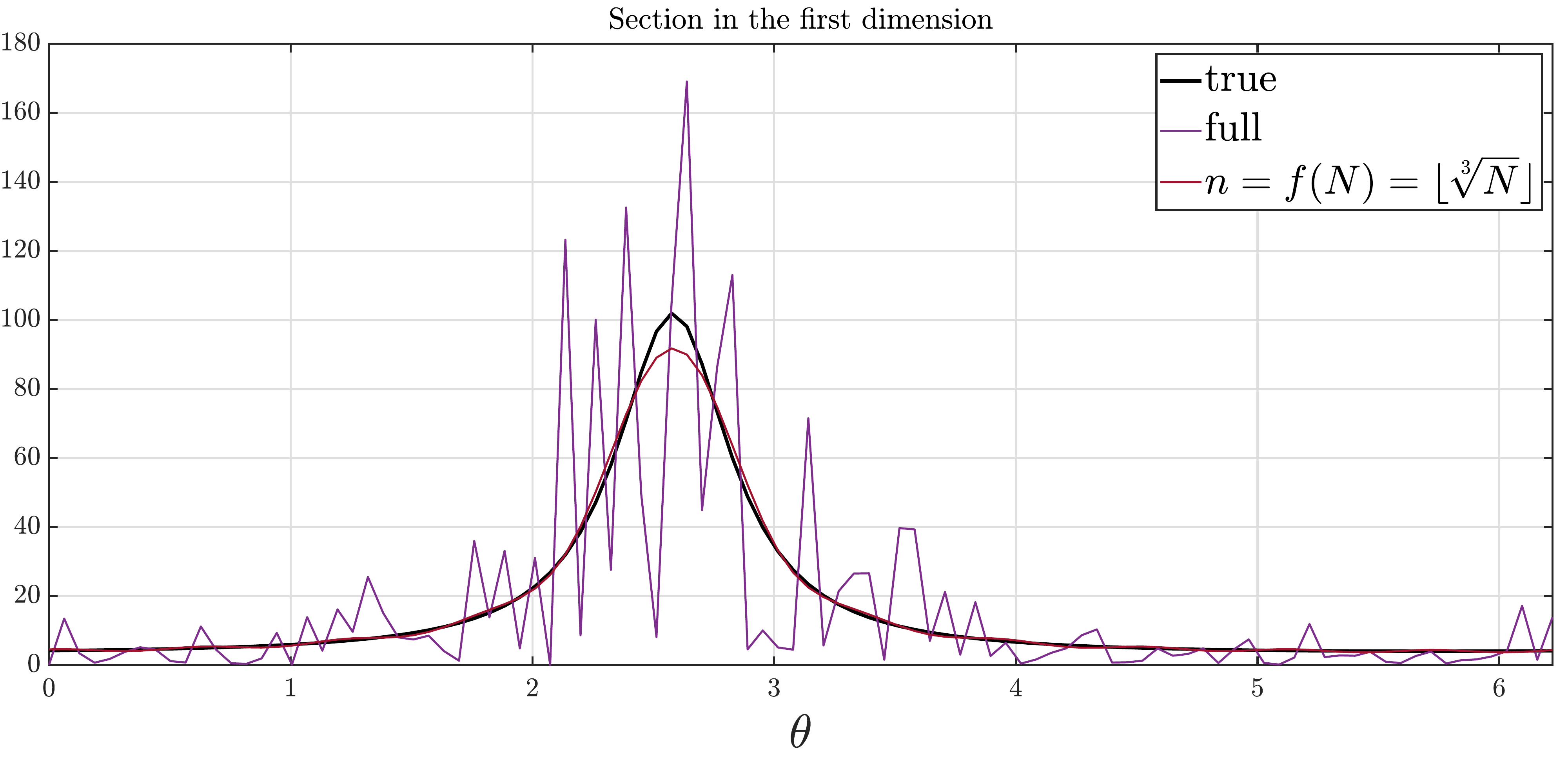}
	\caption{Automotive radar: true and estimated spectra at the cross section $[ \cdot, \; 1.07, \; 2.88 ] $ where $\Phi$ has the peak.}
	\label{fig:radar_sec1}
\end{figure}
\begin{figure}
	\centering
	\includegraphics[width=\linewidth]{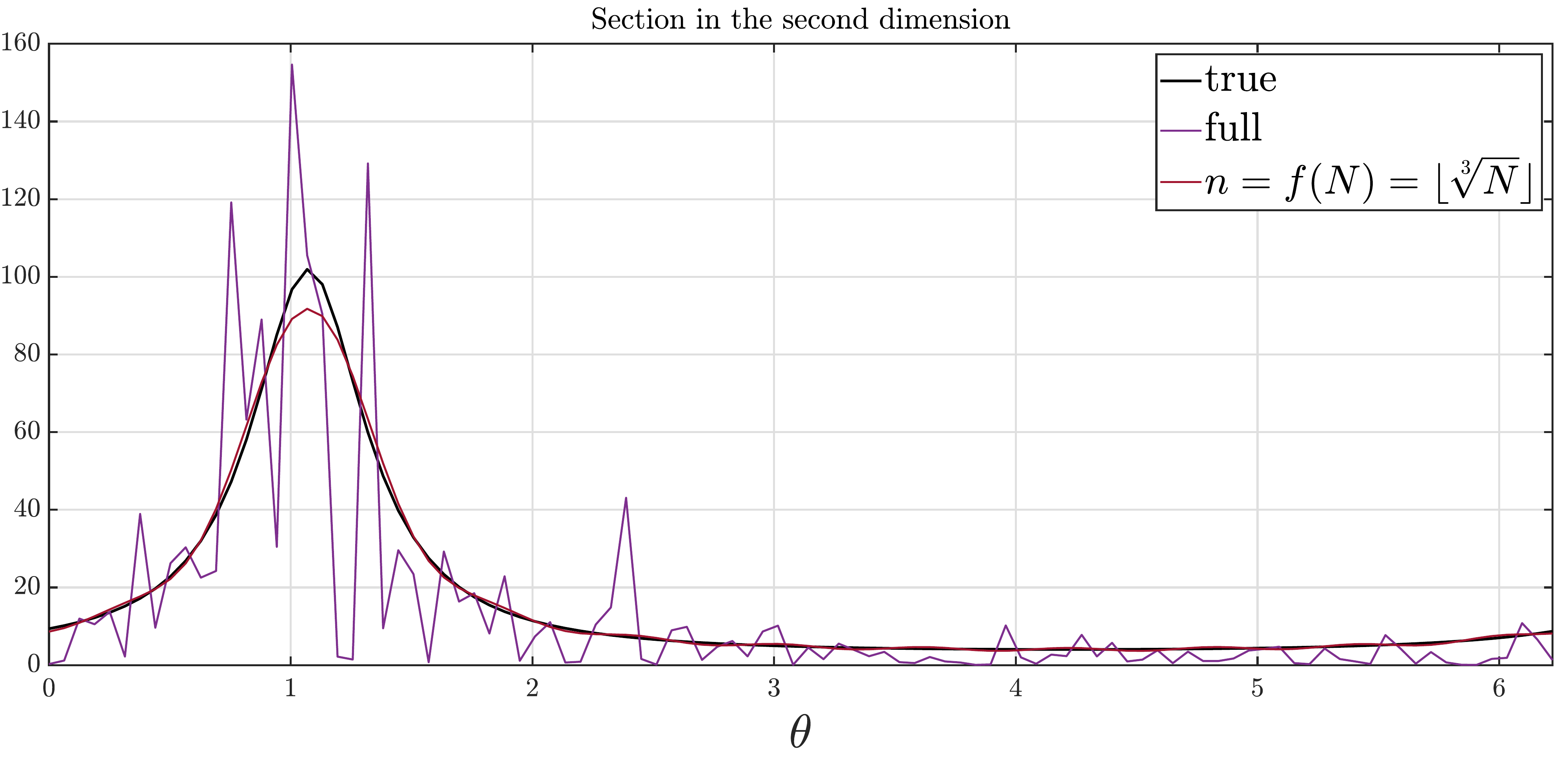}
	\caption{Automotive radar: true and estimated spectra at the cross section $[ 2.58, \; \cdot , \; 2.88 ] $ where $\Phi$ has the peak.}
	\label{fig:radar_sec2}
\end{figure}
\begin{figure}
	\centering
	\includegraphics[width=\linewidth]{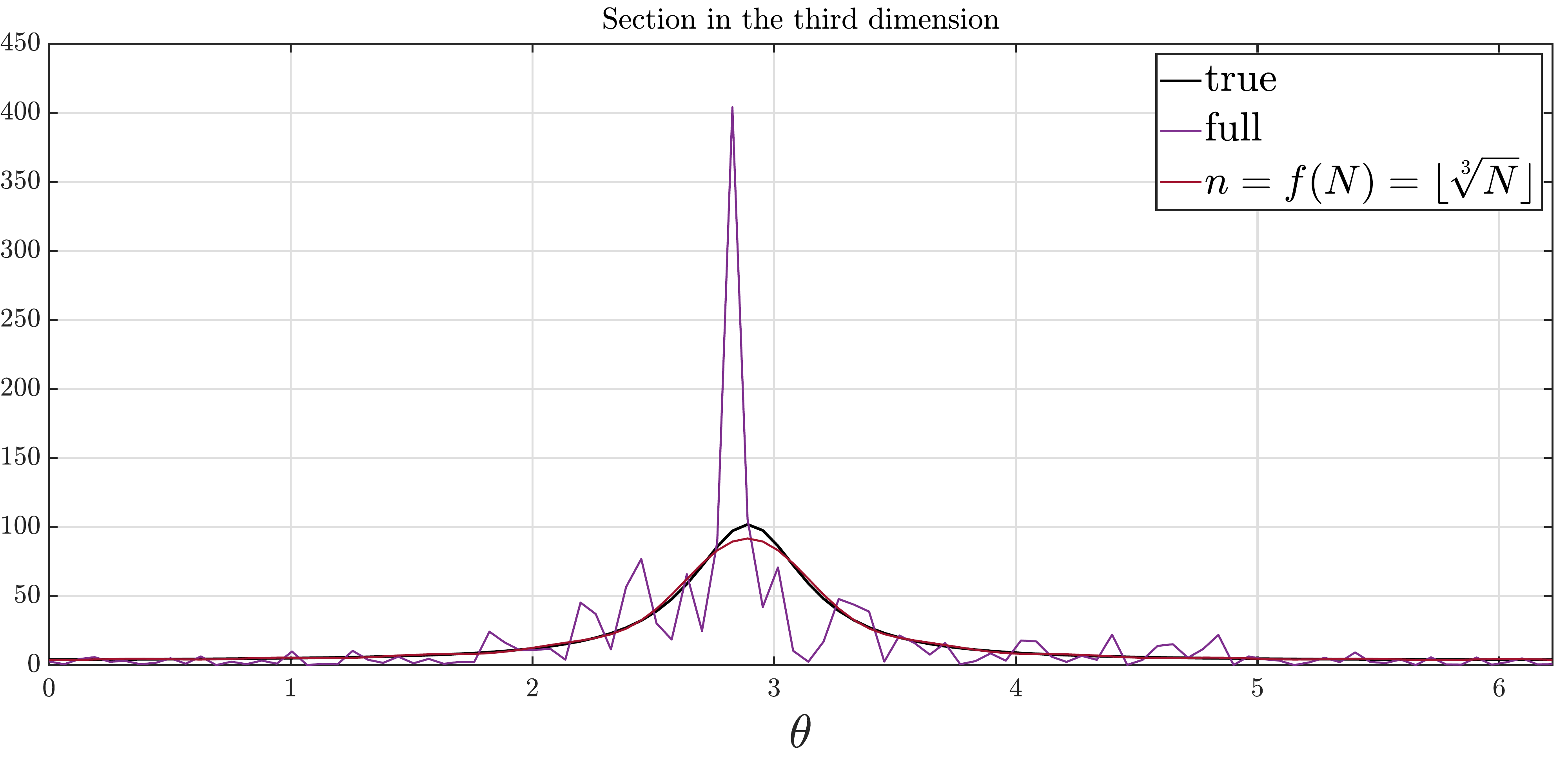}
	\caption{Automotive radar: true and estimated spectra at the cross section $[2.58, \; 1.07, \; \cdot] $ where $\Phi$ has the peak.}
	\label{fig:radar_sec3}
\end{figure}

\section{Conclusions}\label{sec:conc}
	In this paper we have considered the estimation of the spectral density of stationary processes via the $f$-truncated periodogram.
	We have proven that, by choosing the truncation point as an appropriate function $f$ of the sample size, the resulting estimator is mean-square consistent both when applied with the unbiased and biased sample covariances.\\
	To illustrate the theory and to highlight the importance in practice of the proposed estimator, we have performed numerical simulations concerning three concrete identification problems: the impulse response estimation of a SISO system, the problem of learning undirected graphical models and the target parameter estimation in automotive radars. The results confirm that our spectral estimator is effective. We finally remark that results similar to those presented in {\luci Subsections} \ref{sec:ETFE}, \ref{sec:NET} and \ref{sec:RAD} have been obtained by considering the $f$-truncated periodograms with {\luci different truncation functions, namely } 
	$f(N) = \lfloor N^{0.48}\rfloor$ and $f(N) = \lfloor \sqrt[4]{N}\rfloor.$

\end{document}